%% file: ExtensionPropertyTridisk28Nov2017arxiv.tex
\documentclass[12pt]{article}

\usepackage{amstext}
\usepackage{amsthm}
\usepackage{amsmath}
\usepackage{amssymb}
\usepackage{latexsym}
\usepackage{amsfonts}
\usepackage{color}
\usepackage{graphicx}

\usepackage[pagebackref,hypertexnames=false, colorlinks, citecolor=black, linkcolor=black, urlcolor=red]{hyperref}

\definecolor{dark_purple}{rgb}{0.4, 0.0, 0.4}
\definecolor{dark_green}{rgb}{0.0, 0.7, 0.0}

\newcommand\black{\color{black}}

\newcommand\KL{{\mathcal K}_\Lambda}

\input def_latex3

\newtheorem{defin}[theorem]{Definition}

\renewcommand\O{\Omega}
\newcommand\hio{H^\infty(\Omega)}

\renewcommand\l{\lambda}

\newcommand\calv{{\mathcal V}}
\newcommand\calvo{{\mathcal V}_0}
\newcommand\calw{{\mathcal W}}
\newcommand\calwo{{\mathcal W}_0}
\newcommand\hiv{{H^\infty(\calv)}}
\renewcommand\S{{\mathcal S}}
\newcommand\cp{Carath\'eodory-Pick\ }
\newcommand\om{\omega}

\numberwithin{equation}{section}
\title{Extensions of bounded holomorphic functions on the tridisk}
\author{\L{}ukasz Kosi\'nski
\thanks{Partially supported by Iuventus Plus grant IP2015 035174}
\and
John E. M\raise.5ex\hbox{c}Carthy
\thanks{Partially supported by National Science Foundation Grant  
DMS 156243}
}

\begin{document}

\bibliographystyle{plain}
\maketitle

{\sc Abstract:} A set $\calv$ in the tridisk $\D^3$ has the polynomial extension property if for every polynomial $p$ 
there is a function $\phi$ on $\D^3$ so that $\| \phi \|_{\D^3} = \| p \|_\calv$ and $\phi |_\calv = p|_\calv$.
We study sets $\calv$ that are  relatively polynomially convex and have the polynomial extension property.
If $\calv$ is one-dimensional, and is either algebraic, or has polynomially convex projections, we show that it is a retract.
If $\calv$ is two-dimensional, we show that either it is a retract, or, for any choice of the coordinate functions, it is the graph of a function of two variables.

\section{Introduction}
\label{seca}

A celebrated theorem of H. Cartan asserts that if $\O$ is a pseudoconvex domain in $\C^d$
and $\calv$ is a holomorphic subvariety of $\O$, then every holomorphic function on $\calv$ extends to a holomorphic
function on $\O$ \cite{car51}. It is not true, however, that every bounded holomorphic function on $\calv$ necessarily 
extends to a bounded holomorphic function on $\calv$ \cite{henpol84, kn10b}. It is even rarer for every bounded holomorphic function to extend to a bounded holomorphic function of the same norm, and when this does occur, there is a special relationship between $\calv$ and $\Omega$, which we seek to explore.

Let $\calv$ be a subset of $\C^d$.
By a {\em holomorphic function on $\calv$} we mean a function $f : \calv \to \C$ with the property that
for every point $\l$ in $\calv$, there is an open ball $B$ in $\C^d$ that contains $\l$, and a holomorphic
function $\phi: B \to \C$ so that $ \phi |_{B \cap \calv} = f |_{B \cap \calv}$.
We shall denote the bounded holomorphic functions on $\calv$ by $H^\i (\calv)$, and equip this space with the
supremum norm:
\[
\| f \|_{\hiv} \ := \ \sup_{\l \in \calv} | f(\l) |.
\]

\begin{defin}
\label{defa1}
Let $\O$ be a bounded domain in $\C^d$, and $\calv \subseteq \O$ be non-empty.
Let $\A$ be a subalgebra of $\hiv$.
We say that $\calv$ has the  $\A$ extension property w.r.t. $\O$  if,  for every $f \in \A$,
there is a function $\phi$ in $H^\i(\O)$ such that $\phi |_\calv = f|_\calv$ and 
\[
\| \phi \|_{H^\i(\O)} \= \| f \|_\hiv .
\]
When $\A =  \C[z_1, \dots, z_d]$, the algebra of polynomials,  we shall call this the polynomial 
extension property.
\end{defin}

We say $\calv$ is a {\em retract} of $\O$ if there is a holomorphic map $r : \O \to \calv$  such that $r |_\calv = {\rm id} |_\calv$.
Clearly any retract has the polynomial extension property, because $ \phi := p \circ r$ gives a norm-preserving extension.
The converse cannot be true without any regularity assumption on $\calv$, because any set that is dense 
(or dense near the distinguished boundary of $\O$) will  trivially have the polynomial  extension property. We shall restrict our attention, therefore, 
to sets that have some form of functional convexity.
We shall say that $\calv$ is a {\em relatively polynomial convex
 subset} of $\O$ if $\overline{\calv}$
is polynomially convex and   $\overline{\calv} \cap \Omega = \calv$.
We shall say that $\calv$ is $\hio$ convex if, for all $\l \in \Omega \setminus \calv$, there exists
a $\phi \in \hio$ such that
\[
| \phi( \l) | \ > \ \sup_{z \in \calv} |\phi(z) | .
\]
\begin{question}
\label{qa1} If $\calv$ is a relatively polynomial convex subset of $\O$ that has the polynomial extension property, must $\calv$ be a retract of $\O$?
\end{question}

In  \cite{agmcvn} it was shown that the answer to Question \ref{qa1} is yes if $\O$ is the bidisk $\D^2$.
We give another proof of this in Theorem~\ref{thmc1}.
\bt
\label{thma1}
A relatively polynomially convex set $\calv \subseteq \D^2$ has the polynomial extension property
if and only if it is a retract.
\et

Not every retract is polynomially convex.
Indeed, suppose $B$ is a Blaschke product whose zeros are dense on the unit circle.
Then $\calv = (z, B(z))$ is a retract whose closure contains $\T \times \overline \D$, so 
its polynomial hull is the whole bidisk. Moreover, any superset of $\calv$ has the polynomial 
extension property trivially (since polynomials attain their supremum), so \eg
$\calv \cup \{ (1/2, w) : w \in \D \}$ is a holomorphic variety with the polynomial extension property.

A more general version of Question \ref{qa1} is the following. We do not know the answer even for $\O$ equal to the bidisk.

\begin{question}
\label{qa2} Is $\calv$  a retract of $\O$ if and only if  $\calv$ is an $\hio$ convex subset of $\O$ that has the $\hiv$  extension property?
\end{question}

Let $\rho$ denote the pseudo-hyperbolic metric on the disk
\[
\rho(z,w) \= \left| \frac{z-w}{1-\bar w z} \right| .
\]
A {\em Kobayashi extremal} for a pair of points $\lambda$ and $\mu$ in a  domain $\O$ is a holomorphic
function $f : \D \to \O$ such that $\lambda$ and $\mu$ are in the range of $f$, 
 and so that $\rho(f^{-1}(\lambda), f^{-1}(\mu))$ is minimized over all holomorphic functions $g: \D \to \O$
 that have $\lambda$ and $\mu$ in their range. 
 A {\em Carath\'eodory extremal} is a map $\phi: \O \to \D$ that maximizes $\rho(\phi(\lambda), \phi(\mu))$.
 
If $\O$ is convex, there is a Kobayashi extremal for every pair of points,
 and  by a theorem of L. Lempert \cite{lem81}, for every Kobayashi extremal
  $f: \D \to \O$ for the pair $(\lambda,\mu)$ there is a 
 Carath\'eodory extremal $\phi: \O \to \D$ for the pair that is a left-inverse to $f$, \ie
$ \phi \circ f = {\rm id} |_\D$ .

 The range of a Kobayashi extremal is called a {\em geodesic}.
A pair of points $\lambda = (\lambda_1, \lambda_2)$ and $\mu = (\mu_1, \mu_2)$ in $\D^2$ is called {\em balanced}
if $\rho(\lambda_1, \mu_1) = \rho(\lambda_2,\mu_2)$. The Kobayashi extremal is unique (up to precomposition with a M\"obius map)
if and only if $\lambda$ and $\mu$ are balanced. A key part of the proof in \cite{agmcvn} was to show
that if a set with the polynomial extension property contained a balanced pair of points, then it contained the entire geodesic
containing these points.
We give a new proof of Theorem~\ref{thma1} in Section~\ref{secc}.

 In \cite{ghw08}, K. Guo, H. Huang and K. Wang proved
 that the answer to Question \ref{qa1} is  yes if
$\O$ is the tridisk, $\calv$ is the intersection of an algebraic set with $\D^3$, and in addition the polynomial  extension operator
is given by a linear operator  $L$  from $H^\i(\calv)$ to $H^\i(\D^3)$. 
(This is called the {\em strong extension property} in \cite{rud69}).
The principle focus of this paper is to examine what happens for the tridisk without the assumption
that there is a linear extension operator.

For the polydisk, all retracts are described by
the following theorem of L. Heath and T. Suffridge \cite{hs81}:
\bt
\label{thma2}
The set $\calv$ is a retract of $\D^d$ if and only if, after a permutation of coordinates, 
$\calv$ is the graph of a map from $\D^n$ to $\D^{d-n}$ for some $ 0 \leq n \leq d$.
\et
Any relatively polynomially convex subset of $\D^3$ with the polynomial extension property is a holomorphic
subvariety (Lemma \ref{lemb1}), so is of dimension $0,1,2,$ or $3$.
The only $3$-dimensional holomorphic subvariety of $\D^3$ is $\D^3$.
The only $0$-dimensional sets with the extension property are singletons (Lemma \ref{lemb2}).
So we just have to consider
the cases when $\calv$ has dimension  $1$  and $2$.
In Section \ref{secd} in Theorems~\ref{thmd1} and \ref{thmg1} we show:
\bt
\label{thma3}
Let $\calv$ be a relatively polynomially convex subset of $\D^3$ that has the polynomial extension property and 
is one dimensional. If $\calv$ is algebraic or has polynomially convex projections, then $\calv$ is a retract of $\D^3$.
\et
In Section \ref{sece} in Theorem~\ref{thme1} we prove:
\bt
\label{thma4}
Let $\calv$ be a relatively polynomially convex subset of $\D^3$ that has the polynomial  extension property and 
is two dimensional. Then either $\calv$ is a retract, or, for each $r = 1,2,3$,  there is a domain $U_r \subseteq \D^2$
and a holomorphic function $h_r : U_r \to \D$ so that
\se\att
\begin{align}\label{eq:thma4}
\calv  =& \{ (z_1, z_2, h_3(z_1, z_2) ) : (z_1, z_2) \in U_3 \}\\\nonumber
=& \{ (z_1, h_2(z_1, z_3) , z_3) : (z_1, z_3) \in U_2 \}\\\nonumber
=& \{ ( h_1(z_2, z_3), z_2, z_3 ) : (z_2, z_3) \in U_1 \} .
\end{align}
\et
If we could show that one of the sets $U_r$ were the whole bidisk, then $\calv$ would be  a retract.
In Section~\ref{secg}, we show that the set
\[
\{ z \in \D^3: z_1 + z_2 + z_3 = 0 \} 
\]
does not have the polynomial extension property, although it does satisfy 
\eqref{eq:thma4}.

In Section~\ref{sech} we look at the spectral theory connections, and show
that a holomorphic subvariety $\calv \subseteq \D^d$ has the $\A$-extension property if and only it has the $\A$ von Neumann property.
Loosely speaking, the $\A$ von Neumann property means that any $d$-tuple of operators that ``lives on'' $\calv$
has $\calv$ as an $\A$ spectral set; we give a precise definition in Def. \ref{defvn}.

\vskip 10pt
In \cite{kmc17} it was shown that the answer to Question~\ref{qa1} 
is yes if $\O$ is the ball in any dimension, or in dimension $2$
if $\O$ is either strictly convex or strongly linearly convex.

There is one domain  for which the answer to Question~\ref{qa1} is known to be no.
This is the symmetrized bidisk, the  set $G := \{ (z+w, zw) : z,w \in \D \}$. 
 In \cite{aly16}, J. Agler, Z. Lykova and N. Young proved 
\begin{theorem}
 \label{thma5} 
 The set $\calv$ is an algebraic subset of $G$ having the $H^\i(\calv)$ extension property if and only if either
 $\calv$ is a retract of $G$, or $\calv = {\mathcal R} \cup {\mathcal D}_\beta$, where
 ${\mathcal R} = \{(2z,z^2) : z \in \D\}$ and ${\mathcal D}_\beta = \{ (\beta + \bar \beta z, z):
 z \in \D \}$, where $\beta \in \D$.
 \et

\section{Preliminaries}
\label{secb}
\subsection{General Domains}

Note that if $\calv \subseteq \O$ is relatively polynomially convex, it is automatically $\hio$ convex.

We shall use the following assumptions throughout this section:

\begin{enumerate}
\item[(A1)]
{\em $\O$ is a bounded domain, and $\calv$ is a relatively polynomially convex
 subset of 
$\O$ that has the polynomial extension property}.
\item[(A2)]
{\em $\calv$ is an $\hio$ convex
 subset of 
$\O$ that has the $\hiv$ extension property}.
\end{enumerate}

The first lemma is straightforward.
\bl
\label{lemb1} If either (A1) or (A2) hold, then  $\calv$ is a holomorphic subvariety of $\O$.
\el
\bp
Under (A1), 
for every point $\lambda$ in $\O \setminus \calv$, there is a polynomial $p_\l$ such that
$|p_\l(\lambda)| > \| p \|_{\hiv}$. Let $\phi_\l$ be the norm preserving extension
of $p_\l$ from $\calv$ to $\O$. 
Then
\[
\calv \= \bigcap_{\l \in \O \setminus \calv} Z_{\phi_\l - p_\l} ,
\]
where we use $Z_f$ to denote the zero set of a function $f$.

Locally, at any point $a$ in $\calv$, the ring of germs of holomorphic functions is Noetherian
\cite[Thm. B.10]{gun2}. Therefore $\calv$ is locally the intersection of finitely many zero zets of functions
in $H^\i(\O)$, and therefore is a holomorphic subvariety.

Under (A2) the same argument works, where now $p_\l$ is in $\hio$ but not necessarily a polynomial.
\ep

The following lemma is a modification of \cite[Lemma 5.1]{agmcvn}.
\bl
\label{lemb2}
If (A1) holds, then 
$\overline{\calv}$ is connected. If (A2) holds, then $\calv$ is connected.
\el
\bp
In the first case, consider the Banach algebra $P(\overline{\calv})$, the uniform closure of the polynomials in $C(\overline{\calv})$.
The maximal ideal space of $P(\overline{\calv})$ is $\overline{\calv}$ \cite[Thm. III.1.2]{gam}.
Assume  $E$ is a clopen proper subset of  $\overline{\calv}$. By the Shilov idempotent theorem \cite[Thm. III.6.5]{gam},
the characteristic function of $E$ is in $P(\overline{\calv})$. For each $n$, there exists a polynomial $p_n$ such that
$| p_n -1| < 1/n$ on $E$, and $|p_n| < 1/n$ on  $\overline{\calv} \setminus E$. 
By the extension property, there are functions $\phi_n$ of norm at most $1+1/n$ in $H^\i(\O)$ that extend $p_n$.
By normal families, there is a subsequence of these functions that converge to a function $\phi$ of norm
$1$ in $H^\i(\O)$ that is $1$ on $E \cap \O$ and $0$ on  $(\overline{\calv} \setminus E) \cap \O$.
Since $E \cap \O$ is non-empty,  by the maximum modulus theorem, $\phi$ must be constant,
a contradiction to $\overline{\calv} \setminus E$ being non-empty.
Therefore $\overline{\calv}$ is connected.

In the second case, if $E$ is a clopen subset of $\calv$, then the characteristic function of $E$ is in $\hiv$,
so has an extension to $\hio$, 
and the maximum modulus theorem yields that $\calv$ is connected. 
\ep

An immediate consequence of Lemma~\ref{lemb2} is that if $\calv$ is $0$-dimensional, then it is a single point.

In Section \ref{seca} we defined the Kobayashi and Carath\'eodory extremals for a pair of points $\l,\mu$ in a set $\O \subseteq \C^d$.
There is also an infinitesimal version, where one chooses one point $\lambda \in \O$ 
and a non-zero vector $v$ in the tangent space of $\O$ at $\l$.
A Kobayashi extremal is then a holomorphic map $f: \D \to \O$ such that $f(0) = \lambda$
and $Df(0) $ points in the direction of $v$ and has the largest magnitude possible
(or any such $f$ precomposed with a M\"obius transformation of $\D$). 
A Carath\'eodory extremal is a holomorphic map $\phi: \O \to \D$ such that $\phi(\l) = 0$ and
$D\phi(\l) [v]$ is maximal (or any such $\phi$ postcomposed with a M\"obius transformation).

More generally, we shall say that \cp data consists of distinct points $\l_1, \dots \l_N$ in
$\O$, and, for each $1 \leq j \leq N$, between $0$ and $d$ linearly independent vectors
$v_j^k$ (thought of as tangent vectors in $T_{\l_j}(\O)$), and correspondingly $N$ points
$w_1, \dots , w_N$ in $\D$ and complex numbers $u_j^k$ (thought of as tangent vectors
in $T_{w_j}(\D)$). A \cp solution to this data is a holomorphic function $\phi: \O \to \D$
such that 
\beq
\phi(\l_j) & \=  &w_j \\
 D \phi(\l_j) v_j^k &=&  u_j^k \quad \forall j,k .
\eeq
We shall say that $\phi$ is a \cp extremal for some data if $\phi$ is a \cp solution, and no function
of $\hio$ norm less than $1$ is a solution.
If $\calv \subseteq \O$, we shall say that the data is contained in $\calv$ if
each $\l_j$ is in $\calv$ and for each $v_j^k$ there is a sequence of points $\mu_n$ in $\calv$ that
converge to $\l_j$ such that
\[
v_j^k \= \| v_j^k\| \,  \lim_{n \to \i} \frac{\l_j - \mu_n}{\|\l_j - \mu_n \|} .
\]

The next theorem is based on an idea of P. Thomas \cite{th03}.
Let $P(K)$ denote the uniform closure of the polynomials in $C(K)$.

\bt
\label{thmb1} Let $\O$ be  bounded, and assume that $\calv \subseteq \O$
has the polynomial extension property.
Let $\phi$ be a \cp extremal for $\O$ for some data.
If $\phi |_\calv$ is in $P(\overline{\calv})$,
 then 
 $\overline{\phi(\calv)}$ contains the unit circle $\T$.
\et
\bp
Assume that $\overline{\phi(\calv)}$ omits some point on $\T$.

 Then there is a simply connected 
star-shaped
open set
 $U $ such that
\[
\overline{\phi(\calv)} \ \subseteq \ \overline{U}  \subsetneq \overline{\D} .
\]
(Take $U = \D \setminus \D((1 + \vare)e^{i \theta}, 2 \varepsilon)$ for 
suitably chosen
$\theta$ and $\varepsilon$).

Let  $h: U \to \D$ be a Riemann map, and $f : \D \to U$ be its inverse.

Consider the \cp problem on $\D$
 \beq
 g : w_j &\mapsto & f(w_j) \\
 g'(w_j) &=& f'(w_j) .
 \eeq
This can clearly be solved by $f$, so has some solution.
But it is well-known that the solution to every extremal
\cp problem on the disk is given by a unique finite Blaschke product. (See for instance \cite[Thms. 5.34, 6.15]{ampi}
or \cite[Sec. I.2]{gar81}), and the range of $f$ is contained in $U$.
So there is also a solution $g$ of norm strictly less than one, which can be taken to be a constant multiple
$r$  of a Blaschke
product.

Then $g \circ h \circ \phi$ is a solution to the original \cp problem on $\calv$, and
$\| g \circ h \circ \phi \| $ in $\hiv$ is   less than or equal to $r$.
Since $\phi \in P(\overline{\calv})$, for each $n$, there is a polynomial $p_n$ in $d$ variables
and a constant $C$ depending on $U$ 
so that $\| \phi - p_n \|_{\hiv} \leq \frac{C}{n}$ and $p_n(\calv) \subseteq (1- \frac{1}{n}) U$.
As $g \circ h$ can be uniformly approximated on $ (1- \frac{1}{n}) \overline{U}$ by a sequence of polynomials $q_n$, the sequence 
$q_n \circ p_n \in \C[z_1, \dots, z_d]$ are polynomials that converge uniformly to $g \circ h \circ \phi$
on $\calv$. Each such polynomial can be extended to a function $\psi_n$ in $\hio$ with
$\| \psi_n \|_{\hio} = \| q_n \circ p_n \|_\hiv$.
Finally, by normal families, a subsequence of $\psi_n$ will converge to a function $\psi$ 
in $\hio$ of norm at most $r$ that solves the original \cp problem. This contradicts the assumption
that $\phi$ was an extremal.
\ep

\subsection{Balanced Points in the polydisk}
Let $\O$ now be the polydisk, $\D^d$. 
The automorphisms of $\D^d$ are precisely the maps
\[
\l \ \mapsto \ (\psi_1(\l_{i_1}), \dots, \psi_d(\l_{i_d}) ) ,
\]
where $(i_1, \dots, i_d)$ is some permutation of $(1,\dots, d)$ and each $\psi_j$ is a
M\"obius map \cite[p.167]{rud69}.
The properties of being a retract, being connected, being relatively polynomially convex, and having the
 polynomial extension property, are all invariant with respect to automorphisms of $\D^d$.
The last assertion is because any polynomial composed with an automorphism is in $P(\overline{\D^d})$,
the uniform closure of the polynomials. We shall often use this to move points to the origin for
convenience.

\begin{defin}
A pair of distinct points $(\l,\mu)$ in $ \D^d$ is called $n$-balanced, for $1 \leq n \leq d$, if, 
for some permutation  $(i_1, \dots, i_d)$  of $(1,\dots, d)$, we have
\[
\rho(\l_{i_1}, \mu_{i_1})   = \dots = \rho(\l_{i_n}, \mu_{i_n}) \geq \rho(\l_{i_{n+1}}, \mu_{i_{n+1}}) \geq \dots \geq 
\rho(\l_{i_d}, \mu_{i_d}) .
\]
We shall say the pair is $n$-balanced w.r.t. the first $n$ coordinates if $(i_1, \dots, i_n) = (1, \dots, n)$.
\end{defin}
If a pair is $n$-balanced, we can always permute the coordinates so that it is $n$-balanced w.r.t. the first $n$ coordinates.
Let $\pi_{n} : \C^d \to \C^n$ be  projection onto the coordinates $z_{1},\dots, z_{n}$.

A pair of points is $d$-balanced if and only if there is a unique Kobayashi geodesic passing through them.
The Carath\'eodory extremal is unique (up to a M\"obius transformation) if and only if the pair is not 2-balanced.
Theorem~\ref{thmb1} has the following important consequence.
\begin{theorem}
\label{thmb4}
Suppose $\calv$ is a set
that has the polynomial extension property with respect to
$\D^d$. Suppose $\calv$  contains  a pair of points  $(\l,\mu)$  that is $n$-balanced
w.r.t. the first $n$ coordinates.
If $\pi_n(\calv)$ is relatively polynomially convex in $\D^n$, then
$\pi_{n} (\calv)$ contains an $n$-balanced disk of the form
\[
\{(\psi_1(\zeta), \dots, \psi_n(\zeta) ): \zeta \in \D \}
\]
for some M\"obius transformations $\psi_1, \dots, \psi_n$.
\end{theorem}
\bp
By composing with an automorphism of $\D^d$, we can assume that 
\[
\rho(\l_{1}, \mu_{1})   = \dots = \rho(\l_{n}, \mu_{n}) \geq \rho(\l_{{n+1}}, \mu_{{n+1}}) \geq \dots \geq 
\rho(\l_{d}, \mu_{d}) ,
\]
 that $\mu = 0$, and that $\l_j \geq 0$ for each $j$.
Let  $\phi(z) = \frac{1}{n}(z_1 + \dots z_n)$.
By the Schwarz lemma, $\phi$ is a Carath\'eodory extremal for the pair $(\l,\mu)$, so by Theorem~\ref{thmb1},
$\overline{ \pi_{n} (\calv )} = \pi_{n} (\overline{\calv})$ contains the unit circle
$\{ (\tau, \dots, \tau) : |\tau | = 1 \}$.
Since  $\pi_{n}(\overline{\calv})$ is polynomially convex, 
$\pi_{n}(\calv)$ contains $\{ (\zeta, \dots, \zeta) : \zeta \in \D \}$.
\ep

There is an infinitesimal version of Theorem~\ref{thmb4}, most conveniently expressed when we use
an automorphism to move the point of interest to the origin. It is proved from Theorem~\ref{thmb1} 
in the same way.
\bt
\label{thmb5}
Suppose $\calv$ has the polynomial extension property with respect to
$\D^d$, and $0 \in \calv$. Suppose there is a non-zero tangent vector $v \in T_0(\calv)$ such that
$$
|v_1| = \dots =  |v_n| \geq |v_{n+1} | \geq \dots \geq |v_d| .
$$
If   $\pi_{n}(\overline{\calv})$ is relatively polynomially convex,
then it
 contains the disk
$$
\{ (\z,\omega_2 \z, \dots , \omega_n \z) : \z \in \D \}
$$ for some unimodular $\omega_2, \dots, \omega_n$.
\et

\section{The bidisk}
\label{secc}

In this section we will take our domain $\O$ to be the bidisk $\D^2$, and make the following assumption 
about $\calv \subseteq \D^2$:

(A3) $\calv$ is relatively polynomially convex  and has the polynomial extension property w.r.t. $\D^2$.

We shall let 
\[
m_a(z) \= \frac{a-z}{1 - \overline{a} z}
\]
be the M\"obius map that interchanges $a$ and $0$.
A subset of $\D^2$ is called balanced if, whenever it contains a $2$-balanced pair of points, it contains the
entire geodesic through these points.


Let 
\beq
R_1 &\=& \{ (z_1, z_2) \in \D^2\ : \ |z_2| < |z_1| \} \\
R_2 &\=& \{ (z_1, z_2) \in \D^2\ : \ |z_1| < |z_2| \} \\
D_\omega &=& \{ (\zeta, \omega \zeta) \ : \ \zeta \in \D \} .
\eeq

A subset of $\D^2$ is called balanced if, whenever it contains a balanced pair of points, it contains the
entire geodesic through these points.

\bt
\label{thmc1}
Assume that $\calv$ is relatively polynomially convex  and has the polynomial extension property w.r.t. $\D^2$.
Then $\calv$ is a retract of $\D^2$.
\et
\bp
Without loss of generality, we can assume $0 \in \calv$.
By Lemma~\ref{lemb1} we know that $\calv$ is a holomorphic subvariety,
and by Lemma~\ref{lemb2}, if it is $0$-dimensional, it is a point,
so we shall assume it is not $0$-dimensional.

Step 1: If $\calvo$ is a connected component of $\calv$, 
and if $\calvo \cap R_1$ and $\calvo \cap R_2$ are both non-empty, then
$\calvo$ contains $D_\omega$ for some unimodular $\omega$.

If $\calvo$ contains a non-zero point $\l$ with $|\l_1| = |\l_2|$, then 
the pair $(0,\l)$ is $2$-balanced, so
by
Theorem~\ref{thmb4}
we get some $D_\omega$ in $\calv$ and therefore in $\calvo$.
Otherwise, since $\calvo$ is connected, there are sequences $(z_n)$ and $w_n$ tending
to $0$ in $\C$, and numbers $a_n$ and $b_n$ in $\overline{\D}$, so that
$(z_n, a_n z_n) $ and $(b_n w_n, w_n)$ are in $\calvo$.
Passing to a subsequence, we can assume that $a_n$ converges to $a$ and $b_n $ converges to $b$.
If $ab = 1$, choose non-zero $\alpha$ and $\beta$ so that $|\alpha| + |\beta| = 1$
and $\alpha + a \beta = 1$.
Otherwise, let $\alpha$ and $\beta$ be any non-zero numbers such that  $|\alpha| + |\beta| = 1$.
Let $\phi(\l) = \alpha \l_1 + \beta \l_2$.
We will show that $\overline{\phi(\calv)}  \supseteq \T$ by 
 Theorem~\ref{thmb1}.

Let $v^1 =  (1, a)^t$  and $v^2 = (b,1)^t$ be tangent vectors at $0$.
We claim that $\phi$ is extremal for the \cp problem on $\D^2$
\se\att
\begin{eqnarray}
\nonumber
\psi(0) &\= & 0 \\
\nonumber
D \psi(0) v^1 &=& \alpha + \beta a \\
\label{eqc10}
D \psi(0) v^2 &=& \alpha b + \beta
\end{eqnarray}
Indeed, if $v^1$ and $v^2$ are linearly independent, then  
\eqref{eqc10} determines that $D \psi(0) = (\alpha\  \beta)$, so $\phi$ is extremal by the
Schwarz lemma. If they are not, which occurs when $ab =1$, then our choice that $\alpha + a \beta = 1$
still yields $\phi$ is extremal (though no longer the unique solution).

So by Theorem~\ref{thmb1}, we get  $\overline{\phi(\calv)}  \supseteq \T$, and since
$\calv$ is relatively polynomially convex, 
with 
\[
\omega \= \frac{\alpha |\beta|}{\beta |\alpha |} 
\]
we get $D_\omega \subseteq \calvo$.

\vs
Step 2: If $D_\omega \subsetneq \calv$, then $\calv = \D^2$.

Let $\l \in \calv \setminus D_\omega$. There exists some point $\mu$ in $D_\omega$
so that $(\l, \mu)$ is 2-balanced, so by Theorem~\ref{thmb4}, 
 $\calv$ contains two intersecting balanced geodesics.
Composing with an automorphism, we can assume that they intersect at $0$, so that $\calv$ contains
$D_\omega$ and $D_\eta$ for two different unimodular numbers $\omega$ and $\eta$.
Now we repeat the argument in the proof of Step 1 with $a = \omega$ and $b = \overline{\beta}$.
Since $ab \neq 1$, we can choose any $\alpha$ and $\beta$ whose moduli sum to $1$, so we
get that $\calv$ contains $D_\tau$ for every unimodular $\tau$. As $\calv$ is a holomorphic variety, it must be all of $\D^2$.

\vs
After a permutation of coordinates, we can now assume that $\calv \subseteq R_1$.
Let $\pi_1$ be projection onto the first coordinate.

\vs
Step 3: If $\calvo$ is a connected component of $\calv$, then for every $z \in \D$, the set $\pi_1^{-1} (z) \cap \calvo$ contains at most one element.

Otherwise $(z,w_1)$ and $(z,w_2)$ are distinct points in $\calvo$. Composing with the automorphism
of $\D^2$ that sends $$
(0,0), (z,w_1), (z,w_2) \ \mapsto \ (z,w_1), (0,0), (0, m_{w_1}(w_2))
$$ respectively,
we are in the situation of Step 1, and hence by Step 1 and Step 2, $\calv = \D^2$.

\vs
Step 4: If $\calv$ is connected, $\calv$ is a retract.

By Step 3, the only remaining case is when $\calv = \{ (z , f(z) ) : z \in U \}$ where
$U \subseteq \D$ and $f: U \to \D$ satisfies $|f(z)| < |z|$ if $z \neq 0$.
Since $\calv$ is a holomorphic subvariety of $\D^2$, we must have $U = \D^2$ and $f$ is holomorphic.

\vs

Step 5: The set $\calv$ has to be connected. 

It is sufficient to consider the case when it is one-dimensional.
By Steps 1 and 3, $\calv$ cannot have any branch points, so must be a disjoint union of single sheets.
It cannot contain any $2$-balanced pairs, or we are done by Step 2. Assuming $0 \in \calv$, this means that, after a permutation of coordinates if
necessary, there is some sheet
$\S = \{(z, f(z)) : z \in \D \}$ in $\calv$, where $|f(z)| < |z|$ if $z \neq 0$, and no point on
$D_\omega \setminus \{ 0 \}$ for any unimodular $\omega$.
By Lemma~\ref{lemc1}, $\S$ must be all of $\calv$, for otherwise
$\calv$ would contain a 2-balanced pair.
\ep

%
%
%

\bl
\label{lemc1}
Suppose $X \subseteq \D^2$ contains the set
$\S = \{(z, f(z)) : z \in \D \}$  where $f : \D \to \D$ is a holomorphic function satisfying $f(0) = 0$,
and $\S$ is not all of $X$.
Then $X$ contains a 2-balanced pair.
\el
\bp
Let $(z_1, w_1)$ be any point in $X \setminus \S$.
Composing with the automorphism $(m_{z_1}, m_{f(z_1)})$ we can assume that
$(0,w_1) \in X \setminus \S$, and $\S = \{ (z, g(z) ) : z \in \D \}$, where
\be
\label{eqc3}
g(z) \= m_{f(z_1)} \circ f \circ m_{z_1} (z) .
\ee

If $X$ has no $2$-balanced pairs, we must have that for all $z$ in $\D$,
\be
\label{eqc2}
\rho(g(z), w_1) \ > \ \rho(z, 0)  .
\ee

Let $ 1 > r > |w_1|$, and consider $\{g(r e^{i \theta}) : 0 \leq \theta \leq 2\pi \}$.
By \eqref{eqc2}, this set must lie outside the pseudohyperbolic disk centered at $w_1$ of radius $r$, and
inside the disk centered at $0$ of radius $r$ by the Schwarz Lemma (since $g(0) = 0$). By the argument principle, this would mean that
$g$ has no zero in $\D(0,r)$, a contradiction. 
\ep

\section{$\calv$ is one-dimensional with polynomially convex projections}
\label{secd}

In this section we take $\O = \D^3$.
We make the following assumption about $\calv \subseteq \D^3$:

(A4) The set $\calv$  has the polynomial extension property with respect to $\D^3$,
is one-dimensional, 
and both $\calv$ and  $\pi(\calv)$ are relatively polynomially convex for every projection $\pi$  onto
two of the  coordinate functions.


\bt
\label{thmd1}
If $\calv$ satisfies (A4), then it is a retract of $\D^3$.
\et

%

We shall prove Theorem~\ref{thmd1} in a series of 3 Lemmas, \ref{lemd1}, \ref{lemd2} and \ref{lemd3}.
Composing with automorphisms of $\D^3$, we can assume without loss of generality that $0 \in \calv$,
and since $\calv$ is a holomorphic subvariety, we can also assume that $0$ is a regular point.
Thus, there are germs $f_2,  f_3$ such that $\calv$ coincides  with $\{(\zeta, \zeta f_2(\zeta),  \zeta f_3(\zeta))\}$ in a neighborhood of $0$. 
Permuting the coordinates we may also assume that $|f_j'(0)|$  are less than or equal to $1$ for each $j$.
 Let $\calvo$ be the component of $\calv$ containing $0$.

Recall that $\pi_{2} : \C^3 \to \C^2$ is projection onto the first two coordinates.

\begin{lemma}
\label{lemd1}
Either $\calvo$ is a retract of $\D^3$, or,
up to a composition with an automorphism of $\D^3$, the set
 $\{(\z, \z):\z \in \mathbb D\}$ is contained in $\pi_{2}(\calvo)$.
\end{lemma}
\bp
 If $|f_j'(0)| =1$ for  $j = 2$ or $3$, we are done by Theorem~\ref{thmb5}.
 So we shall assume that  they are all less than $1$.
 Let
 \[
 R \ := \ 
 \{ \l \in \D^3 : |\l_j | < |\l_1| \  {\rm for\ } j=2,3 \}.
 \]
If $\calvo \setminus \{ 0 \}$ is not contained in $R$, then there is a point in $\calvo \cap \partial R$ that
is $2$-balanced with respect to $0$, so 
we are finished by Theorem~\ref{thmb4}.
So assume that $\calvo \setminus \{ 0 \}$ is contained in $R$, and that $\calvo$ is not single-sheeted over the
first coordinate,
so it contains $z = (z_1, z_2, z_3)$ and $(z_1, w_2,  w_3)$ where $(z_2, z_3) \neq (w_2, w_3)$.
Moreover, we can assume that $\calvo$ is regular at $z$, since the singular points are of dimension $0$.

Composing with an automorphism that interchanges $0$ and $(z_1, z_2,  z_3)$, we can assume that
$\calvo$ contains $0, z$ and a point $ \mu = (0, \mu_2, \mu_3)$,
and that there is a sheet
 $\{(\zeta, \zeta g_2(\zeta),   \zeta g_3(\zeta))\}$ passing through $0$ and $z$ that
 stays inside $R$.
 By Lemma~\ref{lemc1}, we conclude that 
 there must be a point in  the sheet $\{ (\zeta, \zeta g_2(\zeta),  \zeta g_3(\zeta))
: \zeta \in \D \}$  that is $2$-balanced with respect to $\mu$.
So again Theorem~\ref{thmb4} finishes the proof.
\ep
 
 \bl
 \label{lemd2}
 Suppose $\{(\z, \z):\z \in \mathbb D\}$ is contained in $\pi_{2}(\calvo)$.
 Then there is a holomorphic $f: \D \to \D$ such that
 $\{ (\z,\z, f(\z)) :\z \in \mathbb D\} \subseteq \calvo$.
 \el
 \bp
 Let $\calw = \{ (\z,w) : (\z,\z, w ) \in \calvo \}$.
This is a one-dimensional variety.
Let $\calwo$ be the connected component of $0$.
If $\calwo$ contains a point in $\{ |\z| = |w| \}$ then $\calvo$ contains a 3-balanced point, and we are done.
We assumed $|f_3'(0)| \leq 1$; if equality obtains, then for some unimodular $\omega$ we would have
$\phi(\l) = \frac{1}{3}(\l_1 + \l_2 + \omega \l_3 )$ would satisfy the hypotheses of
 Theorem~\ref{thmb1}, and polynomial convexity would again give a 3-balanced disk in $\calvo$.

So we can assume $\calwo \subseteq \{ |w| < | \zeta| \}$.
Either $\calwo$ is single sheeted over $\z$, and we are done, or as in Lemma~\ref{lemd1}
we find two distinct regular points $(\z,w_1)$ and $(\z, w_2)$. Composing with the automorphism
$\alpha = (m_{\z}, m_{w_1})$, we get the points $(0,0), (\z, w_1) $ and $(0, m_{w_1}(w_2))$ all in 
$\alpha(\calwo)$, and
by Lemma~\ref{lemc1}
we get $\calwo$ contains a $3$-balanced pair.
\ep

\bl
\label{lemd3}
Suppose $\{(\z, \z):\z \in \mathbb D\}$ is contained in $\pi_{2}(\calvo)$.
 Then there is a holomorphic $f: \D \to \D$ such that
 $\{ (\z,\z, f(\z)) :\z \in \mathbb D\} = \calv$.
 \el
 \bp
 By Lemma~\ref{lemd2}, we have 
  $\S := \{ (\z,\z, f(\z)) :\z \in \mathbb D\}$ is a  subset of $ \calv$. Suppose the containment is proper,
  and there exists $(z_1,z_2,z_3) \in \calv \setminus \S$. 
  If $z_2 = z_1$, let  $\calw = \{ (\z,w) : (\z,\z, w ) \in \calv \}$.
This contains the set $ \{ ( \z, f(\z) ) : \z \in \D\}$ and a point $(z_1, z_3)$ with $z_3 \neq f(z_1)$.
By Lemma~\ref{lemc1}, this means $\calw$ contains a $2$-balanced pair, which means $\calv$ contains
a $3$-balanced pair. Since $\calv$ is relatively polynomially convex, by Theorem~\ref{thmb4} this means
$\calv$ contains a $3$-balanced disk.
Since we are assuming that $\calv$ is larger than a disk, there must be another point, and
hence a 2-balanced disk through this point and the $3$-balanced disk.
So, after an automorphism, we can assume that $\calv$ 
contains $\{ (\z,\z,\z) : \z \in \D \}$ and $\{ (\eta, \omega \eta, g(\eta) ) : \eta \in \D \}$
for some unimodular $\omega$.
If $\omega \neq 1$, then for any $\alpha, \beta$ with $|\alpha| + |\beta| = 1$,
the function $\phi(z) = \alpha z_1 + \beta z_2$ will be a \cp extremal 
for the \cp data
\se\att
\begin{eqnarray}
\nonumber
\psi(0) &\= & 0 \\
\nonumber
D \psi(0) (1 \ 1 \ 1)^t &=& \alpha + \beta  \\
\nonumber
D \psi(0) (1 \ \omega \  g'(0) )^t &=& \alpha  + \beta  \omega\\
D \psi(0) (0\  0\  1)^t &=& 0.
\label{eqd1}
\end{eqnarray}
So by Theorem~\ref{thmb1}, $\phi(\overline{\calv})$ will contain $\T$,
and by polynomial convexity, this means that
$\pi_2 \calv = \D^2$. Hence $\calv $ could not have been one dimensional.
(Indeed, since $\pi_2$ is  Lipschitz, it cannot increase Hausdorff dimension).
If $\omega = 1$, then $g'(0) \neq 1$, so we interchange the second and third coordinates and repeat the argument.

 If $z_2 \neq z_1$, then there will be some point in $\S$ that is $2$-balanced with respect to 
$z$, so by Theorem~\ref{thmb4} $\calv$ contains a $2$-balanced disk in addition to $\S$.
Repeating the previous argument again shows that $\calv$ cannot be two-dimensional.
\ep

\section{$\calv$ is one dimensional and algebraic}
\label{secg}

We shall say that $\calv \subseteq \D^d$ is algebraic if  there is a set of polynomials
 such that  $\calv$ is the intersection of $\D^d$
with their common zero set.
(The set can always be chosen to be finite by the Hilbert basis theorem.)
Let $\calw$ be the common zero set of the polynomials in $\C^d$
(so $\calv = \calw \cap \D^d$).

\bt
\label{thmg1}
 Suppose:
 (A5) The set $\calv$ is a one-dimensional algebraic subset of $\D^3$ that has the polynomial extension property.
 
 Then it is a retract.
 \et
 
 First, we prove that it is smooth. We need to use the following four results.
The first one  is  \cite[Thm. 5.4A]{wh72}

\begin{proposition}
\label{prop:whit}

Let $V,W$ be analytic spaces and $F:V\to W$ be proper and extend to be continuous and holomorphic on $\overline{V}$
(\ie $F$ is c-holomorphic). Then $F(V)$ is analytic in $W$. 
\end{proposition}

The next result is from \cite[p. 122]{chi90}.
\begin{proposition}\label{prop:chir}
Let $A$ be an analytic set in $\mathbb C^n$, $a\in A$, $dim_a A=p$. Assume that there is a connected neighborhood $U=U'\times U''$ of $a$ such that $\pi:U\cap A \to U' \subseteq \C^p$ is proper.

Then there exist an analytic set $W\subset U'$, $\dim W <p$, and $k\in \mathbb N$ such that
\begin{itemize}
 \item $\pi: U\cap \pi^{-1}(W) \to U'\setminus W$ is a local $k$-sheeted covering;
 \item $\pi^{-1}(W)$ is nowhere dense in $A_p\cap U$, where $A_p=\{z\in A:\ \dim_z A=p\}$.
\end{itemize}

\end{proposition}

The following proposition essentially can be found in proofs that are scattered over Sections 3.1 and 3.2 in \cite{chi90}. For the convenience of the reader we recall its (elementary) proof.
\begin{proposition}\label{prop:prop}
Let $A$ be an analytic set in an open domain $\Omega\subset \C^p\times \C^m$ and $\pi:(z',z'')\to z'$ be a projection onto $\C^p$. If $a=(a', a'')$ is an isolated point of $\pi^{-1}(a')\cap A$, then there is a polydisc $U=U'\times U''$ with the center equal to $a$ such that $\pi:U\cap A\to U'$ is proper.
\end{proposition}

\bp One can find a polydisc $U''$ such that $\overline{ U''}\cap \pi|_A^{-1}(a) = \{a''\}$. Since $A$ is closed, there is a polydisc $U'$ such that $A$ does not have limits points on $U'\times \partial U''$, which means that $\pi:U\cap A\to U'$ is proper.
\ep

The next tool that will be exploited in the present section is taken from \cite[Chap. V.1]{loj91}.
\black
\begin{proposition}[The analytic graph theorem]\label{prop:graph}
	Let $U$, $V$ be complex manifolds and $f:U\to V$ locally bounded. Then $f$ is holomorphic if and only if its graph $\{(x, f(x)):\ x\in U\}$ is analytic in $U\times V$.
\end{proposition}

We are now in position to start the proof of the theorem.
\bl
\label{lemg4}
If (A5) holds, and there is an automorphism $\Phi$ of $\D^3$ so that 
$\Phi(\calv) \subseteq \D^2 \times \{  \eta \} $ for some $\eta \in \D$,
then $\calv$ is a retract.
\el
\bp
Under the hypotheses, $\pi_2(\Phi(\calv))$ is a polynomially convex subset of $\D^2$ that
has the polynomial extension property
 so by Theorem~\ref{thma1}
it is a retract of the bidisk. It follows that  $\Phi(\calv)$, and hence $\calv$,  are retracts of the tridisk.
\ep

%

\bl
\label{lemg2}
 Assume (A5) holds, and 
  $\calv$ contains a $2$-balanced pair $(w',w'')$  that is not $3$-balanced.  Assume also that $w''$ is a regular point of $\calv$.
  Then there is an automorphism $\Phi$ of $\D^3$ that takes $w'$ to $0$ and $w''$ to
  $w= (w_1, w_1, w_3)$ and an irreducible component $\calv'$ of $\calv$ 
  such that $\Phi(\calv')$ contains $0$ and $w$ and such that
  $\calv' \subseteq \{ z_1 = z_2 \}$.
%
 \el

 \bp
After an automorphism, we can assume that
$\calv$ contains $0$ and a point
$w = (w_1, w_1, w_3)$ with $|w_3| < |w_1|$.
By Lemma~\ref{lemb2}, the point $w$ is not isolated in $\calv$.
Consider the function
\beq
\psi : \calv \times \C &\  \to \ & \C \\
(z, \z) &\mapsto & z_1 - \zeta z_2 .
\eeq
This function vanishes at $(w,1)$, and therefore on a one-dimensional subvariety of $\calv \times \C$
in a neighborhood of that point. If it vanishes on the set $\zeta =1$, then
$z_2 - z_1$ vanishes on all of some component $\calv'$ of $\calv$.

Otherwise, by the Weierstrass preparation theorem,
there is  a non-empty open set $E$ of the circle $\T$ so that $\calv$ intersects
$$
W_\om \  := \
\{ |z_3| < | z_2 | = |z_1|\}  \cap \{ z_2 = \omega z_1 \}
$$  for every $\omega \in E$.
Let $F_\omega$ be any inverse to the map $ \zeta \mapsto (\z, \omega \z)$, and define
\beq
G_\omega : \D^3 &\to & \D \\
z &\mapsto &
 F_\omega(z_1 , z_2 ) .
 \eeq
For each $\omega \in E$, there is a point $w = (w_1, \omega w_1, w_3)$ in $W_\om$,
and there is a  geodesic in $\D^3$ that contains $0$ and $w$. The function $G_\om$ is a left inverse
to the Kobayashi extremal through these points. By Theorem~\ref{thmb1}, $\overline{G_\om(\calv)}$ contains the unit circle.
Therefore $\pi_2(\overline{\calv})$ contains $\T \times E$.
But since $\calv$ is algebraic, there can only be finitely many points lying over any point in $\T$, except perhaps for a
zero-dimensional singular set.
\ep

\bl\label{lemg21}
Let $f_i, g_i$, $i=1,2,$ be holomorphic functions in the closed unit ball in $H^{\infty}(\D (t))$. Let $\calv$ be an analytic variety in $\D^3$ that contains two discs $\mathcal S=\{(\lambda, \lambda f_1(\lambda), \lambda f_2(\lambda)):\ \lambda\in \D(t)\}$ and $\mathcal S'=\{(\lambda g_1(\lambda), \lambda, \lambda g_2(\lambda)):\ \lambda \in \D(t)\}$. 
If the  germs of these discs at $0$ are not equal, then one can find two points, one in $\mathcal S\setminus\{0\}$ and the second in $\mathcal S'\setminus\{0\}$, that are arbitrarily close to $0$ and form a $2$-balanced pair.
\el

\bp
Let us consder the values 
\begin{equation}\label{eq:21} \quad \rho(\lambda, \mu g_1(\mu)),\ \rho(\lambda f_1(\lambda), \mu),\ \rho(\lambda f_2(\lambda), \mu g_2(\mu)).
\end{equation}
If the inequality $\rho(\lambda, \mu g_1(\mu))\leq\rho(\lambda f_1(\lambda), \mu)$ is satisfied for $\lambda, \mu \in \D(s)$, where $s>0$ is small, then $f_1$ is a unimodular constant (to see it take $\mu=0$) and $g_1f_1=1$ (put $\lambda f_1 =\mu$). If additionally $ \rho(\lambda f_2(\lambda), \mu g_2(\mu))\leq\rho(\lambda f_1(\lambda), \mu),$ $\lambda,\mu\in \D(s)$, putting $\mu=\lambda f_1$ we find that $\mathcal S$ and $\mathcal S'$ coincide near $0$. This shows that for $\lambda$ and $\mu$ ranging within $\D(s)$ the maximum of the values in \eqref{eq:21} cannot be attained by the first term listed there. By symmetry, the same is true for the second term and a similar argument shows that values in \eqref{eq:21} cannot be dominated by the third term, as well. 

Consequently, allowing $\lambda$ and $\mu$ to range within $\D(s)\setminus\{0\}$ we see that the maximum of three hyperbolic distances in \eqref{eq:21} is attained by at least two of them simultaneously. For such a choice of $\lambda$ and $\mu$ the points $(\lambda, \lambda f_1(\lambda), \lambda f_2(\lambda))$ and $(\mu g_1(\mu), \mu, \mu g_2(\mu))$ form the 2-balanced pair we are looking for.
\ep

Let $\B(t)$ be the polydisc $\D^3(t)$ of radius $t$ centered at the origin.

 \bl
 \label{lemg1}
 If (A5) holds, then $\calv$ is locally a graph of a holomorphic function.
\el

\bp 
Since the property is local it suffices to show that $\mathcal V$ is smooth at $0\in \calv$.
 
Any analytic set is a locally finite union of its connected components. Therefore we can choose $t >0$ so that any irreducible component of $\mathcal V$ that intersects $\B(t)$ contains $0$.
 
For each $j=1,2,3$, write $\mathcal V$ as the union of two analytic sets $\mathcal W_j\cup \mathcal V_j$ such that $\mathcal W_j$ is contained in $\{z_j=0\}$ while $0$ is an isolated point of $\calv_j\cap \{z_j=0\}$.
 
Let $\pi_j:\C^3\to \C$ denote the projection on the $j$-th variable, $z\mapsto z_j$, $j=1,2,3$.
Decreasing $t$ we can assume that $\pi_j|_{U_j\cap \calv_j}\to U_j'$ is proper for some polydisc $U_j=U_j'\times U_j''$ containing $\B(t)$ (Proposition~\ref{prop:prop}) and that any point of $\mathcal V_j\cap U_j$, possibly without $0$, is a regular point of $\calv$.
Let $W_j\subset U_j'$ be as in Proposition~\ref{prop:chir}. Since it is a discrete set, decreasing $t$ we can also assume that $W_j$ and $\D(t)$ have at most one common point and that the common point is $0$, if it exists, $j=1,2,3$.
 
{\bf Claim 1.} Assume that there is a point $x$ in $\B(t)\cap \calv$ such that $|x_1|>|x_2|,|x_3|$. Then, near $0$ the variety $\calv_1$ is a graph $\{(\lambda, \lambda f(\lambda), \lambda g(\lambda)):\ \lambda\in \D(t)\}$ for some $f,g$ in the open unit ball of $H^{\infty}(\D(t))$.

\begin{proof}[Proof of Claim 1] 
To prove the assertion we need to show that $\calv_1$ is single sheeted near $0$, that is the multiplicity of the projection $$\pi:U_1\cap \calv_1\to U_1'$$ is equal to $1$. Actually, this would mean that in a neighborhood of $0$ the variety is $\mathcal V_1$ is of the form $\{(\lambda, \lambda f(\lambda), \lambda g(\lambda)):\ \lambda\in \D(t)\}$ for some functions $f,g$ that are locally bounded, and thus holomorphic, according to Proposition~\ref{prop:graph}. For any $z\in\mathcal V_1\cap U_1$ the pair $(0,z)$ is not $2$-balanced. If it were not the case, we could use Lemma~\ref{lemg2} to find that $f$ and $g$ are unimodular constants, which contradicts the assumption that $|x_j|<|x_1|$, $j=1,2$. Consequently $|f(\lambda)|, |g(\lambda)|<1$, $\lambda\in \D(t)$, proving the assertion.

Since $x_1\notin W_1$ it is enough to show that $\mathcal V_1\cap U_1$ is single sheeted over $x_1$. Suppose the contrary, that is we can find another point $x'=(x_1, y_2, y_3)$ in $\calv \cap U_1$.
 Let $\gamma$ be a curve in $\calv\cap U$ joining $0$ and $x'$ and such that $\gamma(t)\neq x$ for any $0\leq t \leq 1$ and $\gamma(t)$ is a regular point of $\calv$ for any $t\in (0,1]$. The automorphism $$\Phi(z)=(m_{x_1}(z_1), m_{x_2}(z_2), m_{x_3}(z_3))$$ switches $0$ with $x$ and the curve $\Phi \circ \gamma$ joins  $(x_1, x_2, x_3)$ with $(0, m_{x_2}(y_2), m_{x_3}(y_3))$. Therefore the curve must meet one of sets $\{|z_3|\leq |z_1|=|z_2|\}$ or $\{|z_2|\leq |z_1|=|z_3|\}$. We lose no generality assuming that the first possibility holds.
 If $\Phi\circ \gamma$ meets a point $w\neq 0$ such that $|w_1|=|w_2|=|w_3|$, then by Theorem~\ref{thmb4} the set $\Phi(\calv )$
 contains a $3$-balanced disk through $0$, in addition to a curve joining $0$ to $x$. This would make
 $0$ a multiple point of $\Phi(\calv)$, so $x$ would be a multiple point of $\calv$, contradicting the assumption
 that it was smooth.
 If $\Phi\circ \gamma$ meets the set $$\Sigma:=\{z\in \C^3:\ |z_3| < |z_1| = |z_2|\}$$ we get a contradiction using Lemma~\ref{lemg2}, since at the first point of intersection of $\Phi\circ \gamma$ with $\Sigma$, say at $w = \Phi \circ \gamma(t_0)$, 
 a neighborhood of $w$ in $\calv$ contains an analytic disk inside the image of $\Sigma$ under $\Phi$.
 This means that $\calv$ is not smooth at $w$.
 \end{proof}
{\bf Claim 2.} Assume that $x=(x_1, x_1, x_3)\in \mathcal V \cap \B(t)$ is such that $0<|x_3|\leq |x_1|$. Then $\mathcal V_1\cap \{z\in \B(t):\ z_1=z_2\}=\{(\lambda, \lambda, \lambda f(\lambda)):\ \lambda\in \D(t)\}$ for some $f$ in the closed unit ball of $H^{\infty}(\D(t))$. 

\begin{proof}[Proof of Claim 2]
Since $(0,x)$ is either a two- or three-balanced pair, the variety $\calw:=\calv_1\cap\{z_1=z_2\}$ is one dimensional at $0$. Repeating the argument used in Claim 1 it is enough to show that $\calw$ is single sheeted over $\{z\in \B(t):\ z_1= z_2\}$. To see it,
take two points $\mu=(\mu_1, \mu_1, \mu_3)$ and $\nu=(\mu_1, \mu_1, \nu_3)$ in $\mathcal V_1\cap \B(t)$.

If either $|\mu_3|>|\mu_1|$ or $|\nu_3|>|\nu_1|$ then, after a  proper permutation of coordinates, we find from Claim~1 that $\mathcal V\setminus \{z\in \C^3_*:\ z_3=0\}$ is a graph of a function over the third variable: $$\{(\lambda f(\lambda), \lambda g(\lambda), \lambda):\ \lambda\in \D(t)\},$$ where $|f|,|g|<1$ on $\D(t)$, which is impossible, as $x$ belongs to it.

If in turn $|\mu_3|=|\nu_1|$, then for some unimodular $\omega$, the variety $\mathcal V$ contains the disc $\{(\lambda, \lambda, \omega \lambda):\ \lambda\in \D\}$. One can find $\lambda\in \D(t)$ such that $$\rho(\mu_1,\lambda) = \rho(\mu_3, \omega \lambda),$$ which entails that there is a 3-balanced disc in $\calv$ passing through $\mu$ and $(\lambda, \lambda, \eta \lambda)$. Consequently, $\calv$ is not smooth at $\mu$. Of course, the same holds if $|\mu_3|=|\mu_1|$.

Finally consider the case when $|\mu_3|<|\mu_1|$ and $|\nu_3|<|\mu_1|$.  Let $\gamma$ be a curve in $\mathcal V'\cap U$ joining $0$ and $\mu$. A continuity argument proves that there is $s>0$ that satisfies the equality $$\rho(\mu_1, \gamma_1(s)) = \rho(\nu_3, \gamma_3(s)).$$ Again, we can obtain a contradiction with the smoothness constructing a balanced disc passing through $\gamma(s)$ and $\mu$. 
\end{proof}

{\bf Claim 3.} Suppose that there is a point $x\in \mathcal V_1\cap \B(t)\setminus\{0\}$ such that $|x_1|\geq|x_2|\geq |x_3|>0$. Then $\calv_1\cap \B(t)$ is the graph of a holomorphic function.

\begin{proof}[Proof of Claim 3]
Note first that the assertion is an immediate consequence of Claim 1 provided that there is a point in $y\in\calv_1\cap \B(t)$ such that $|y_1|>\max(|y_2|, |y_3|)$. On the other hand, if there is a point $y$ in $\calv_1\cap \B(t)$ that satisfies $|y_{2}|>\max(|y_{1}|,|y_{3}|)$ or $|y_3|>\max(|y_1|, |y_2|)$, then Claim 1 gives a contradiction. 

Therefore, we need to focus  on the case when any $y\in \calv_1\cap \B(t)$ satisfies $|y_1|=|y_2|\geq |y_3|$ or $|y_1|=|y_3|\geq |y_2|$. Note that $\calv_1\cap \{z_3=0\}$ is discrete. Thus Claim~2 provides us with a description of intersections $\calv_1$ with the hyperplanes 
$$l_1=\{z_1=\omega z_2\}, \quad l_2=\{z_1= \omega z_3\}$$
for unimodular constants $\omega$. In particular, if $\calv_1$ lies entirely in one of these hyperplanes, we are done. Otherwise, there are at least two points in $\calv_1\cap \B(t)\setminus \{0\}$ that lie in two different hyperplanes. Applying Claim~2 (after a proper permutation of coordinates and multiplication of them by unimodular constants) we find that $\calv_1$ contains two different analytic discs. 
The possibilities that may occur here are listed below. The first one describes the case when both points lie in hyperplanes of type $l_1$ (or type $l_2$, after a change of coordinates) while the second one refers to the case when one of the points is in $l_1$ and the second in $l_2$:
\begin{enumerate}
	\item[i)] $(\lambda, \lambda, \lambda f(\lambda)), (\lambda, \omega \lambda, \lambda g(\lambda))\in \calv$ for $\lambda\in \D(t)$,
\item[ii)] $(\lambda, \lambda, \lambda f(\lambda)), (\lambda, \lambda g(\lambda), \lambda)\in \calv$ for $\lambda\in \D(t)$,
\end{enumerate}
where $f$ and $g$ are in the closed unit ball of $H^\infty(\D(t))$ and $\omega\in \T$.
Making use of Lemmas~\ref{lemg21} and \ref{lemg2} we see that both cases contradict the smoothness of $\calv$ outside the origin.
\end{proof}

{\bf Claim 4} If $\calw_1$ is not discrete, then it is the graph of a holomorphic function.
\begin{proof}[Proof of Claim 4] If there is a point $(0,y_2, y_3)$ in $\calw_1$ such that $|y_2|\neq |y_3|$ we are done, as after a proper change of coordinates Claim 1 can be applied here. Otherwise, for some unimodular $\omega$ there is a disc of the form $\{(0, \lambda, \omega \lambda): \lambda\in \D\}$ that is entirely contained in $\calv$. If there are two different discs in $\calv$ we end up with a particular case of possibility i) that occurred in Claim 2 (take $f$ and $g$ equal to $0$ and and multiply the coordinates by unimodular constants).
\end{proof}

We come back to the proof of Lemma~\ref{lemg1}. If there is a point $x\in \calv\cap \D(t)\setminus \{0\}$ 
all of whose  coefficients do not vanish, then $\calv$ is a union of at most two graphs of holomorphic functions, due to Claims~3 and 4. If there is no such point, then $\calv\cap \D(t)$ can also be expressed as at most three graphs, according to Claim~4. 

If $\calv$ is not a graph of one function, then permuting coordinates we see that it contains two discs $$\{(\lambda, \lambda f(\lambda), \lambda g(\lambda)):\ \lambda \in \D(t)\},\quad \{(0, \lambda, \lambda h(\lambda)):\ \lambda\in \D(t)\},$$ where $f,g,h$ are in the closed unit ball of $H^\infty(\D(t))$.
Here, again, Lemmas~\ref{lemg2} and \ref{lemg21} give a contradiction.
\ep

{\sc Proof of Theorem~\ref{thmg1}.}
We have shown that $\calv$ is smooth.
If $\calv$ contains a $3$-balanced pair, then it contains a $3$-balanced disk by Theorem~\ref{thmb4}.
If this is all, then it is a retract. If it contains any other point, then that point and some point in the $3$-balanced
disk would form a $2$-balanced pair that is not $3$-balanced, and we get a contradiction from Lemma~\ref{lemg2}.

So we can assume that $\calv$ contains no $3$-balanced pairs, and, by Lemma~\ref{lemg2} again,
no $2$-balanced pairs either, or else it would be a retract.  After an automorphism,
we can assume that $0 \in \calv$ and and 
\be
\label{eqg1}
\calv \setminus \{ 0 \} \  \subseteq \  \{ \max( |z_2|, |z_3| )< |z_1 | \} .
\ee
In a neighborhood of $0$, we can write $\calv$ as $\{ (\zeta, f_2 (\zeta), f_3(\zeta) ) \}$ for some holomorphic
functions $f_2$ and $f_3$ that vanish at $0$. 
Apply Proposition~\ref{prop:chir} with $U=\D^2\times \D$ and
$A = \calv$.
By \eqref{eqg1}, the projection
so that $\pi_1 : \calv \to \D$ is proper.
Thus we get that $\calv$ is locally $k$-sheeted over $\D$, except over possibly
finitely discrete set of points. But since $\calv$ is smooth, and squeezed by \eqref{eqg1}, we must have $k=1$.
Therefore $f_2$ and $f_3$ extend to be holomorphic from $\D$ to $\D$, and
$\calv$ is a retract.
\ep

\black

\section{$\calv$ is two-dimensional}
\label{sece}

\begin{theorem}
\label{thme1}
 Let $\mathcal V$ be a polynomially convex 2-dimensional analytic variety that has the extension property in $\mathbb D^3$.

Then, either $\mathcal V$ is a holomorphic retract or, for any permutation of the coordinates, $\mathcal V$ is of the form $\{(z_1, z_2, f(z_1, z_2)):\ (z_1, z_2)\in D\}$, where $D\subset \D^2$ and $f\in \mathcal O(D)$. 

\end{theorem}

Throughout this section let $\pi_{ij}:\C^3\to \C^2$ denote the projection onto $(z_i, z_j)$ variables.

The following lemma may be seen as the infinitesimal version of Theorem~\ref{thmb1}.
\begin{lemma}\label{leme1} 
1. Suppose that there there is a sequence $\{(t_n, \gamma_n t_n, \delta_n t_n)\}$ in $\calv$ converging to $0$ such that $\gamma_n\to \gamma_0\in \T$ and $\delta_n\to \delta_0\in \T$. Then $\{(\zeta, \gamma_0 \zeta, \delta_0 \zeta):\ \zeta\in \D\} \subset \calv$.

2. Suppose that there are two sequences $\{(t_n, t_n, \delta^j_n t_n)\}$ in $\calv$ converging to $0$ such that $\delta^j_n\to \delta^j_0$, $j=1,2$, and $\delta_0^1 \neq \delta_0^2$. Then $\{(\zeta, \zeta, \eta):\ \zeta\in \D, \eta\in \D\} \subset \calv$.
\end{lemma}

\begin{proof}
1. It is enough to prove the lemma for $\gamma_0 = \delta_0=1$. Assume that the assertion is not true. Then we can find a triangle $\Gamma$ in $\mathbb D$ with one vertex on $\T$ such that $F(\calv)\subset D:=\D\setminus \bar\Gamma$, where $F(z) = (z_1 + z_2 + z_3)/3$. Let $\Phi_D:D\to \mathbb D$ be a mapping fixing the origin such that $\Phi_D'(0)>1$. Let $G:\mathbb D^3\to \mathbb D$ be a holomorphic extension of $\Phi_D \circ F$. It is clear that $G(0)=0$, so it follows from the Schwarz lemma that $|G_{z_1}'(0)| + |G_{z_2}'(0)| + |G_{z_3}'(0)|\leq 1$. Now dividing the equality $\Phi_D(F(t_n, \gamma_n t_n, \delta_n t_n)) = 
G(t_n, \gamma_n t_n , \delta_n t_n)$ by $t_n$ and letting $n\to \infty$ we find that $\Phi_D'(0) = G_{z_1}'(0) + G_{z_2}'(0) + G_{z_3}'(0)$; a contradiction.

2. We proceed as in the previous part, with the exception that we take $F(z) =\alpha z_1 + \alpha z_2 + \beta z_3$, where $\alpha$ and $\beta$ are any complex numbers satisfying $2|\alpha| + |\beta | = 1$. Again, what we need is to show that $\overline{F(\calv)}$ contains the unit circle. 
	
Suppose that the assertion is not true, \ie $F(\calv)\subset D:=\D\setminus \Gamma$, where $\Gamma$ is a triangle chosen analogously to before. With $t = \Phi_D'(0)>1$ and
$G$ a norm-preserving extension of $\Phi_D \circ F$, we get
\[
G_{z_1}'(0) + G_{z_2}'(0) + \delta_0^j G_{z_3}'(0) = t (F_{z_1}'(0) + F_{z_2}'(0) + \delta_0^j F_{z_3}'(0)) .
\]
 From this system of equations we get that $G_{z_1}'(0) + G_{z_2}'(0) = 2 t\alpha$ and $G_{z_3}'(0) = t \beta$. This again
   contradicts the Schwarz lemma.
\end{proof}

\begin{lemma}\label{leme2}
Let us suppose that $0$ is a regular point of $\mathcal V$, $\dim_0\calv=2$, and that, in a neighborhood of $0$,  $\calv$ is given by $$\{(z_1, z_2, f(z_1, z_2))\},$$ where $f$ is a germ of an analytic function at $0$. Let  $\alpha_j = f_{z_j}'(0)$. 

If $\omega\in \T$ is such that $|\alpha_1 + \omega \alpha_2|\leq 1$, then $\mathcal V$ is single sheeted over $\{z_2=\omega z_1 \}$.

\end{lemma}


\begin{proof}
Step 1: In the first step we shall show that for any $\omega\in \T$ satisfying the assertion there is a holomorphic function $\varphi:\D\to \D$, $\varphi(0)=0$, such that an analytic disc $\zeta \mapsto (\zeta, \omega\zeta, \varphi(\zeta))$ is contained in $\calv$.

Let us suppose that $\omega$ is such that $|\alpha_1 + \omega \alpha_2 |= 1.$ Then it follows from Lemma~\ref{leme1} that $\{(\zeta, \omega \zeta, (\alpha_1 + \omega \alpha_2)\zeta):\ \zeta\in \D\}\subset \mathcal \calv$.

So we are left with the case $|\alpha_1 + \omega \alpha_2|<1$. Note that $\mathcal \calv\cap \{z_2=\omega z_1\}$ is a one-dimensional analytic variety. 
Let $\calw_0$ be its irreducible component containing $0$. Two possibilities need to be considered: either $\calw_0\setminus \{0\}$ intersects $\{(z_1, \omega z_1, z_3):\ |z_1|=|z_3|\}$, or $\calw_0\setminus\{0\}$ is contained in $\{(z_1, \omega z_1, z_3):\ |z_3|<|z_1|\}$. In the case of the first possibility, $0$ and a point of the intersection form a $3$-balanced pair, whence the disc $\zeta\mapsto (\zeta, \omega \zeta, \eta\zeta)$ is contained in $\calv$ for some unimodular $\eta$, which  contradicts the local description of $\mathcal \calv$ near $0$ (precisely, the assumption that $|\alpha_1 + \omega \alpha_2|<1$).

Assume that the second possibility holds. Then the projection $\pi_1:(z_1, z_2, z_3)\mapsto z_1$ restricted to the variety $\calw_0$ is proper. Consequently, $\pi_1(\calw_0)$ is a one-dimensional variety in $\D$, whence $\pi_1(\calw_0)=\D$, by Proposition~\ref{prop:whit}. Therefore it suffices to show that $\calw_0$ is single sheeted over $\{(\zeta, \omega \zeta):\ \zeta\in \D\}$. Actually, once we get it we shall be able to express $\calw_0$ as $\{(\zeta, \omega \zeta, \varphi(\zeta): \zeta\in \D\}$, where the function $\varphi$ is, in particular, bounded, and thus holomorphic, according to Proposition~\ref{prop:graph}.

To prove that $\calw_0$ is single sheeted take $(z_1, \omega z_1, z_3)$ and $(z_1, \omega z_1, w_3)$ in $\calw_0$. Recall that $|z_3|,|w_3|<|z_1|$. Let us compose $\calw_0$ and points 
with an idempotent automorphism $\phi$ of $\D^3$ interchanging $0$ and $(z_1, \omega z_1, z_3)$. Then  $0\in \phi(\calw_0)$, $z=(z_1, \omega z_1, z_3)\in \phi(\calw_0)$ and $y=(0, 0, m_{z_3}(w_3))\in \phi(\calw_0)$. If $\phi(\calw_0)\setminus \{0\}$ intersects $\{(z_1, \omega z_1, z_3):\ |z_3|=|z_1|\}$ we can find a 3-balanced pair in $\phi(\calw_0)$,
 which implies that the disc $\zeta\mapsto (\zeta, \omega \zeta, m(\zeta))$ is in $\calw_0$ for some M\"obius map $m.$
Otherwise we can find in $\phi(\calw_0)$ two sequences $(x_n, \omega x_n, a_n x_n)$ and $(y_n, \omega y_n, b_n y_n)$ converging to $0$ such that $|a_n|<1$ and $|b_n|>1$. Using Lemma~\ref{leme1} we get 
\[
\phi(\calw_0) \supseteq \{ (m_{z_1}(\zeta), m_{\omega z_1} (\omega \zeta), m_{z_3}(\delta \zeta)) : \zeta \in \D \},
\]
which means that
\[
\calw_0 \supseteq \{ (\z, \omega \z, m_{z_3}(\delta m_{z_1}(\z))) : \z \in \D \} ,
\]
as claimed. Since $\calw_0\setminus\{0\}\subset \{|z_3|<|z_1|\}$ we find that $m_{z_3} (\delta m_{z_1}(\z) = \eta \z$, $\z\in \D$, for some unimodular $\eta$. This is in a contradiction with the desription of $\calv$ near $0$.

Step 2: Now we shall prove the assertion, that is we shall show that $ \calv$ is single-sheeted over $\{z_2 = \omega z_1\}$. Seeking a contradiction suppose that $(z_1, \omega z_1, z_3)\in \calv$ and $\varphi(z_1)\neq z_3$, where $\varphi$ is an analytic disc contructed in Step 1. Take $\lambda$ such that 
$\rho(\lambda, z_1) = \rho(\varphi(\lambda), z_3)$. To justify that such a $\lambda$ exists note that it is trivial if $\varphi$ is not proper (consider the values above for $\lambda =z_1$ and properly chosen $\lambda$ close to the unit circle). On the other hand if $\varphi$ is a proper seflmapping of the unit disc, then it is a Blaschke product, so $\varphi^{-1}(z_3)$ is non-empty. Then, considering the values for $\lambda=z_1$ and $\lambda'$ that is picked from $\varphi^{-1}(z_3)$ the desired existence of $\lambda$ follows. Then, of course, $(\lambda, \omega \lambda, \varphi(\lambda))$ and $(z_1, \omega z_1, z_3)$ form a 3-balanced pair, which means that there is a geodesic $\zeta\mapsto (\zeta, \omega \zeta, m(\zeta))$, where $m$ is a M\"obius map, contained in $\calv$ and intersecting $\zeta \mapsto  (\zeta, \omega \zeta, \varphi(\zeta))$ exactly at one point. Thus Lemma~\ref{leme1}, applied at the point of intersection, implies that $\{(\zeta, \omega \zeta, \eta ):\ \zeta\in \D, \eta\in \D\}$ is contained in $ \calv$, which gives a contradiction with the local description of $\mathcal V$ near $0$.
\end{proof}

\begin{corollary}\label{cor:leme2}
Keeping the assumptions and notation from Lemma~\ref{leme2}: if $|\alpha_1|+ |\alpha_2|\leq  1$, then $\mathcal V$ is an analytic retract.
\end{corollary}
\begin{proof}
It follows from Lemma~\ref{leme2} that $f$ extends to $\Delta:=\{(z_1, z_2)\in \D^2:\ |z_1|=|z_2|\}$ and any $x\in\Delta\times \D$ is an isolated point of $\pi_{12}^{-1}(\pi_{12}(x))\cap \calv$. 
Applying Propositions \ref{prop:prop} we find that $\pi_{12}|_{\calv}$ is proper when restricted to a small neighborhood of $x$. 
Thus, by Proposition \ref{prop:chir}, $\pi_{12}|_{\calv}$ is a local $k$-sheeted covering near $x$. Since any analytic set containing $\Delta$ is two dimensional, we get that $k=1$. Consequently, $f$ extends holomorphically to a neighborhood of $\Delta$, due to Proposition~\ref{prop:graph}. 

 Take a Reinhardt domain $R_{\Delta}\subset \mathbb D^2$ that contains $\Delta$ such that $f\in \mathcal O(R_{\Delta})$. The envelope of holomorphy of $R_\Delta$, denoted $\widehat{R_\Delta}$, is a Reinhardt domain, as well.
Since $R_\Delta$ touches both axis, we infer that the envelope is complete, meaning $(\lambda_1 z_1, \lambda_2 z_2)\in \widehat{R_\Delta}$ for any $\lambda_1, \lambda_2\in \mathbb D,$ and $(z_1,z_2)\in \widehat{R_\Delta}$. Consequently, $\widehat{R_\Delta}=\D^2$, as $\Delta\subset R_\Delta$. Therefore, $f$ extends holomorphically to the whole bidisc.

Since the distinguished boundary of $r \D^2$ (equal to $r\T^2$), $r<0<1$, is contained in $\Delta$ we get that $|f|<1$ on $r\D^2$ for any $0<r<1$. Consequently, $f$ lies in the open unit ball of $H^\infty (\D^2)$.
\end{proof}

\begin{corollary}\label{cor:singsheet}
Let us assume that $0$ is a regular point $\mathcal V$ and that its germ near $0$ is of the form 
\begin{equation}\label{eq:cor} \{(z_1, z_2, f(z_1, z_2))\}.
\end{equation}
Let us denote $\alpha_j:=f_{z_j}'(0)$. Then $\mathcal V$ is an analytic retract if one of the possibilities holds:
\begin{itemize}
\item $|\alpha_1| + |\alpha_2|\leq 1$,
\item $|\alpha_1|\geq  1 +|\alpha_2|$,
\item $|\alpha_2|\geq 1 + |\alpha_1|$.
\end{itemize}

If $\mathcal V$ is not an analytic retract, then the set of $2$-dimensional regular points of $\calv$ is single sheeted in each direction. 

Moreover, for any $x\in \calv_{reg}$ there are two pairs of unimodular constants $(\omega_i, \eta_i)$ such that the analytic disc $\{(\z, \omega_i \z, \eta_i \z):\ \z\in \D\}$ lie in $\phi(\calv)$, where $\phi=(m_{x_1}, m_{x_2}, m_{x_3})$ is an indempotent automorphism switching $0$ and $x$.
\end{corollary}

\begin{proof} The first case is covered by Corollary~\ref{cor:leme2}, while the other two are obtained simply by permuting the coordinates.

To prove the second part, when $\calv$ is not an analytic retract, choose a point $x$ in $\mathcal V_{reg}$ such that $\dim_x \calv =2$. We want to show that $\pi_{ij}^{-1}(\pi_{ij}(x))\cap \mathcal V =x$ for any choice of coordinates $(z_i, z_j)$. 
We can make two simple reductions: composing with an automorphism of the tridisc we can assume that $x=0$, and we can focus only on the coordinates $(z_1, z_2)$.

Since $0$ is a regular point we can express $\calv$ as in \eqref{eq:cor}. Since $\calv$ is not a holomorphic retract none of the inequalities listed in the statement of the corollary is satisfied. This, in particular, means that there are two unimodular constants $\omega_i$ such that $|\alpha_1 + \omega_i \alpha_2|=1$. It follows from Lemma~\ref{leme2} that $\calv$ is single sheeted over the $\{z_2=\omega_i z_1\}$, whence over $0$, as well.

The proof of Step~1 in Lemma~\ref{leme2} shows that $\eta_i:=\alpha_1 + \omega_i \alpha_2$, $i=1,2,$ satisfy the last assertion of the corollary.
\end{proof}

\begin{lemma}
\label{leme3}
Suppose that $\mathcal V$ is not an analytic retract. Let $\mathcal W$ be its 2-dimensional connected component (which means that $\mathcal W$ is a connected component of $\calv$ and $\dim_x \mathcal W=2$ for some $x\in \mathcal W$). Then
\[
\mathcal W = \{(z_1, z_2, f(z_1, z_2)): (z_1, z_2)\in D\}, 
\]
where $D$ is an open subset of $\D^2$ and $f\in \mathcal O(D, \D)$.

This can be done over each pair of coordinate functions.
\end{lemma}

\begin{proof}
Let $\calw_0$ be a strictly $2$-dimensional component of $\mathcal V$ (\ie the union of its two dimensional irreducible components in $\calw$). We shall show that $(\calw_0)_{sing}$ is empty. 

Proceeding by contradiction, take a point $a\in (\calw_0)_{sing}$. Suppose first that $\dim_a (\calw_0 \cap (\{(a_1, a_2)\}\times \D)) =0$, which means that $\pi_{12}^{-1}(a_1, a_2)\cap \calw_0 \cap U = \{a\}$ for some neighborhood $U$ of $a$.
According to Proposition~\ref{prop:prop} the projection $\pi_{12}|_A$ is proper in a neighborhood of $a$. Since $\mathcal V_{reg}$ is single sheeted by Corollary \ref{cor:singsheet}, we see that $\pi_{12}$ is a single sheeted covering near $a$ (it is a covering according to Proposition~\ref{prop:chir}). The fact that the covering is single-sheeted immediately implies that $\calw_0$ is smooth there -- a contradiction.

 Permuting coordinates, we trivially get from the above reasoning the following statement: $\dim_b(\mathcal W_0\cup(\D\times \{(b_2,b_3)\})) =0$, where $b\in \mathcal W_0$, implies that $b$ is a regular point of $\calw_0$.

So we need to show that $\dim_a (\mathcal \calw_0 \cap (\{(a_1, a_2)\}\times \D)) =0$. If it were not true, \ie $\dim_a (\mathcal \calw_0 \cap (\{(a_1, a_2)\}\times \D)) =1$, we would be able to find a disc $\Delta$ centered at $a_3$ such that $\{(a_1, a_2)\}\times \Delta \subset \calw_0$, whence $\{(a_1, a_2)\}\times \D \subset \calw_0$. 
Since $\mathcal V$ is single sheeted over its regular points we find that $\{(a_1, a_2)\}\times \D \subset\calv_{sing}$, and thus $\{(a_1, a_2)\}\times \D \subset (\calw_0)_{sing}$. 

 Here we can again permute coordinates in the preceding argument --- note that we are able to do it because $\dim_{(a_1,a_2,x)} (\calw_0 \cap( \D \times\{(a_2, x)\})) =1$. In this way we find that $\D\times \{(a_2, x)\}\subset (\calw_0)_{sing}$ for any $x\in \D$. Consequently, $(\calw_0)_{sing}$ is 2-dimensional, which is impossible.

Thus we have shown that $\calw_0$ is smooth, so it is locally a graph. According to Corollary~\ref{cor:singsheet} for any $x\in \calw_0$ that is a regular point of $\calv$, the variety $\calw_0$ is in a neighborhood of $x$,
 a graph over each pair of coordinate functions. 
In particular, none of the inequalities involving derivatives from that corollary (understood after an automorphism) is satisfied at $x_0$, and by the continuity none is satisfied at points $x\in \calw_0\cap \calv_{sing}$ (if there are any), as well. Therefore $\calw_0$ is a graph over every choice of the coordinate functions for any $x\in \calw_0$.

To prove that $\calw_0=\mathcal W$ we  proceed by  contradiction. Assume that there is $x\in \calw_0$ that lies in the analytic set $\mathcal W'$ composed of $1$ dimensional irreducible components of $\calw$. Then $x$ is an isolated point of $\calw'\cap \calw_0$. Let us take $a\in \calw'$ that is sufficiently close to $x$. Changing  coordinates we can suppose that $(a_1,a_2)\neq (x_1, x_2)$. Then $\mathcal V$ is smooth at the point of the intersection
of  $\calw$ and $\pi_{12}^{-1}(\pi_{12}(a))$ (note that $\pi_{12}(\calw_0)$ is open), and thus it is single-sheeted over $\pi_{12}(a)$; a contradiction.
\end{proof}

\begin{proof}[Proof of Theorem~\ref{thme1}]
Suppose that $\calv$ is not an analytic retract. Let $x\in \calv$ be such that $\dim_x\calv =2$. Then, it follows from Lemma~\ref{leme3} that the connected component $\calw$ containing $x$ is one of the forms listed in \eqref{eq:thma4}. Therefore, to prove the assertion we need to show that $\calv$ is connected.

Choose $x\in \calv\setminus \calw$. We can make a few helpful assumptions. First of all, according to Corollary~\ref{cor:singsheet}, it can be assumed that an analytic disc $\{(\lambda, \lambda, \lambda):\ \lambda\in \D\}$ lies entirely in $\calw$. Changing, if necessary, the coordinates we can also assume that there is a point $\lambda_0$ such that $$\rho(\lambda_0, x_1) = \rho(\lambda_0, x_2)\geq \rho(\lambda_0, x_3).$$
 Then we can compose $\calv$ with the automorphism $\Phi$ of the tridisc that interchanges $0$ and $(\lambda_0, \lambda_0, \lambda_0)$ to additionally get that $|x_1|=|x_2|\geq |x_3|$. Let $\omega\in \T$ be such that $x_2= \omega x_1$. Since, by Corollary~\ref{cor:singsheet}, $\calv$ is single sheeted over each point of $\pi_{12}(\calw)$, we are done, provided that $(x_1,\omega x_1)\in \pi_{12}(\calw)$. Suppose that it is not true.

Let us consider two values
$\rho(\lambda, x_1) = \rho(\omega \lambda, x_2) $ and $\rho(f(\lambda, \omega \lambda), x_3)$. If $\lambda$ moves from $0$ in the direction $x_3$,
 then near the first point $\lambda'\in \D$ such that $(\lambda', \omega \lambda')\notin\pi_{12}(\calw)$ the last value tends to 1. Since the second value is smaller for $\lambda=0$,
  we find that there is some $a$ such that $(a, \omega a)\in \pi_{12}(\calw)$ and  the
  two points $(a, \omega a, f(a, \omega a))$ and $x$ form a $3$-balanced pair. In particular, they can be connected with a $3$-geodesic that entirely lies in $\calv$; a contradiction.
\end{proof}

\section{Further properties and examples}
\label{secg}

Let $\calv$ be a relatively polynomially convex set in $\D^3$ that has the extension property and is not a retract. So far two crucial properties have been derived in the preceding section:
\begin{enumerate}
\item[a)] for each choice of the coordinate functions $\calv$ is a graph of a holomorphic function;
\item[b)] for any $x\in \calv$ there exist two pairs of unimodular constants $(\omega_i, \eta_i)$, $i=1,2$, such that $\{(\z, \omega_i \z, \eta_i \z):\ \z\in \D\}$ lies entirely in $\Phi(\calv)$, $i=1,2,$ where $\Phi$ is an idempotent automorphism of $\D^3$ interchanging $0$ and $x$.
\end{enumerate}

\begin{exam}
\label{exg1}
Observe that $$\calv_0:=\{z\in \D^3:\ z_3= z_1+ z_2\}$$ satisfies a) and b). We shall show that $\calv_0$ does not have the extension property. 

Let $U:=\{(z_1, z_2)\in \D^2:\ |z_1+z_2|<1\}$. Let $\varphi_m$ denote the M\"oubius map $\varphi_m(\zeta) = \frac{m-\zeta}{1 -  \bar m \zeta}$, where $m\in \D$. 

Let us put $h(z_1, z_2):=z_1 + z_2$ and observe that there are points $\z$ and $\xi$ in the unit disc such that $(\z, \z \varphi_m(\z))$, $(\xi, \xi \varphi_m(\xi))$ lie in $U_3$ and
\begin{equation}\label{eq:ex}
\rho\left( \frac{h(\z, \z \varphi_m(\z))}{\z}, \frac{h(\xi, \xi \varphi_m(\xi))}{\xi} \right) < \rho(\z, \xi).
\end{equation} Indeed, it suffices to take $\z$ and $\xi$ sufficiently close to $1$ such that $\rho(\z, \xi)$ is big enough.

Note that \eqref{eq:ex} implies that there is $\psi_m\in \mathcal O (\mathbb D, \mathbb D)$ such that both points $(\z,\z\varphi_m(\z), \z\psi_m(\z)), (\xi ,\xi \varphi_m (\xi), \xi\psi_m(\xi))$ lie in $\calv$.

Let us consider a $3$-point Pick interpolation problem
$$\begin{cases}
0\mapsto 0,\\
(\z,\z\varphi_m(\z), \z\psi_m(\z))\mapsto \z\varphi_m(\z),\\
(\xi ,\xi \varphi_m (\xi), \xi\psi_m(\xi))\mapsto \xi \varphi_m (\xi).
\end{cases}
$$
Note that one solution to the above problem is the function
 $$F(z_1, z_2, z_3)=(z_1\varphi_m(z_1) + z_2)/2.$$
  Observe that it is also extremal (see \cite{kosthree}). Indeed, otherwise we would be able to find a holomorphic function $G$ on the tridisc, with the range relatively compact in $\D$, such that
\begin{equation}\label{eq:ex2}G(x, x \varphi_m(x), x \psi_m(x)) = x \varphi_m(x)
\end{equation}
for $x=0,\z,\xi$. Since $\varphi_m$ is a M\"obius map we find that \eqref{eq:ex2} holds for any $x$, contradicting the fact that that $G(\D^3)\subset \subset \D$. 

Consequently, $F$ interpolates extremally, whence $\mathbb T\subset F(\overline \calv)$, by Theorem~\ref{thmb1}. Thus there is a point $z\in \overline\calv$ such that $F(z)=1$, which means that $z_1 \varphi_m(z_1) =1$ and $z_2=1$. 
Note that $(1,1)\notin \overline{U}$.
Now we easily get a contradiction, as the M\"obius map $\varphi_m$ satisfies the following property: any solution of the equation $x \varphi_m(x) = 1$, $x\in \D$, is close to $1$ as $m$ approaches $1$.
\end{exam}

\begin{remark} The argument from  this example can be applied to the algebraic case. To be more precise suppose that $\calv$ is an algebraic set with the extension property that is not a retract. Write $h:=h_3$ and $U=U_3$, where $h_3, U_3$ are as in Theorem~\ref{thma4} and $h(0,0)=0$. We shall also write $(x,y,z)$ for the coordinates $(z_1, z_2, z_3)$. Note that $|h|$ extends continuously to $\overline U$.
	
Repeating the idea from the example we can show that: if $(-1,1)\in \overline U$ is such that $|h(-1,1)|<1$, then $(1,1)\in \overline U$. Using transitivity of the group of automorphisms of the polydisc we can show slightly more. Choose $\omega,\eta\in \T$ such that $(\z, \omega \z, \eta\z)\in \calv$ for $\z\in \D$. Put $\Psi_a(x,y) = (\varphi_a(z), \varphi_{\omega a}(y))$, $x,y\in \D^2$, $a\in \D$. Let $\tilde h = \varphi_{\eta a} \circ h \circ\Psi_a$ and $\tilde \calv=\{(x,y)\in \Psi_a(U):\ z= \tilde h(x,y)\}$. Note that $|\tilde h(\varphi_a(-1), \varphi_{\omega a}(1))|<1$, so $(-\varphi_a(-1), \varphi_{\omega a}(1))\in \overline{\Psi_a(U)}$. Consequently, $(\varphi_a(-\varphi_{a}(1)),1)\in \overline U$ for any $a\in \D$. Thus we have shown that $(\omega, 1)\in \overline U$ for any $\omega\in \T$.
\end{remark}

\begin{remark}
	If we apply the previous remark to the case when $h$ is rational, we get that either $U$ is the whole bidisc or $h(\T^2)\subset \T$ (whenever it makes sense). The simplest class of such functions contains among others $$h:(z_1, z_2)\mapsto \omega\frac{A z_1 + B z_2 + z_1 z_2}{1 + \bar B z_1 + \bar A z_2},$$ where $\omega\in \T$ and $A$, $B$ are complex numbers. Observe that if $|A|+|B|\leq 1$, then $h$ is defined 
 on the whole bidisc. Thus
we are interested in the question whether for complex numbers $A,B$ such that $|A|+|B|>1$ and $|A|,|B|\leq 1$ the surface 
	\begin{equation}\label{ex:eq} \mathcal V:= \left\{(x,y,z)\in \D^3:\ z =\omega \frac{A x + By + xy}{1 + \bar B x + \bar A y}\right\},
	\end{equation}
	$\omega\in \T$,	has the extension property.
\end{remark}

\begin{remark}
	It is interesting that \eqref{ex:eq} appears naturally in another way. Namely,  it is the \emph{uniqueness variety} for a three-point Pick interpolation problem in the tridisc. 
	
	To explain it, take $\alpha,\beta, \gamma$ in the unit disc that are not co-linear and let $\delta$ be a strict convex combination of these points. For fixed $x,y\in \D$, $x\neq 0$, $y\neq 0$, $x\neq y$, let us consider the following problem:
	$$\D^3\to \D,\quad 
	\begin{cases}0\mapsto 0,\\ (x \varphi_{\alpha} (x), x \varphi_{\beta}(x), x \varphi_{\gamma}(x))\mapsto x \varphi_{\delta}(x),\\
	(y \varphi_{\alpha} (y), y \varphi_{\beta}(y), y \varphi_{\gamma}(y))\mapsto y \varphi_{\delta}(y).
	\end{cases}
	$$
It is an extremal three point Pick interpolation problem. Moreover, it was shown in \cite{kosthree} that the problem is never uniquely solvable, but  there is a set on which  all solutions do coincide, namely all interpolating functions are equal on the real surface $\{(\z \varphi_{t\alpha} (\z), \z \varphi_{t\beta}(\z), \z \varphi_{t\gamma}(\z)):\ \z\in \D, t\in (0,1)\}.$ This in particular means that the uniqueness variety contains points 
\begin{equation}\label{ex:point}\left( \frac{\alpha t\z - \z^2}{1 - \bar \alpha t \z}, \frac{\beta t\z - \z^2}{1 - \bar \beta t \z}, \frac{\gamma t\z - \z^2}{1 - \bar \gamma t \z}\right),
\end{equation}
where $t$ and $\z$ run through an open subset of $\C^2$ (containing $(0,1)\times \D$). Some computations, partially carried out in \cite{kosthree}, show that the set composed of points \eqref{ex:point} coincides with the variety \eqref{ex:eq} with properly chosen $\omega$, $A$ and $B$.
\end{remark}







\section{Von Neumann Sets and Spectral Theory}
\label{sech}

There is a connection between the extension property and the theory of spectral sets for $d$-tuples of operators.
Let $T = (T_1, \dots, T_d)$ be a $d$-tuple of commuting operators on some Hilbert space $\h$. We shall call $T$ an
{\em And\^o $d$-tuple} if 
\[
\| p(T) \| \ \leq \ \sup \{  |p(z) | : z \in \D^d \} \quad \forall\ p \in \C[z_1, \dots, z_d ] .
\]
Let $\calv$ be a holomorphic subvariety of $\D^d$. We shall say that a commuting $d$-tuple $T$ is {\em
subordinate to $\calv$} if $\sigma(T) \subset \calv$ and, whenever $g$ is holomorphic on a neighborhood of $\calv$
and satisfies $g|\calv = 0$, then $g(T) = 0$.  
If $f$ is any holomorphic function on 
$\calv$, then by Cartan's theorem $f$ can be extended to a function $g$ that is holomorphic not just on
a neighborhood of $\calv$ but on all of $\D^d$, and if $T$ is subordinate to $\calv$ then $f(T)$ can be defined
unambiguously as equal to $g(T)$.

Let $\A \subseteq \hiv$ be an algebra, and assume $T$ is subordinate to $\calv$.
 We shall say that $\calv$ is an $\A$-spectral set for $T$ if
 \[
\| f(T) \| \ \leq \ \sup \{  |f(z) | : z \in \calv \} \quad \forall\ f \in \A .
\]
\begin{defin}
\label{defvn}
If $\calv$ is  a holomorphic subvariety of $\D^d$, and $\A \subseteq \hiv$,
 we say $\calv$ has the $\A$ von Neumann property if, whenever $T$ is an  And\^o $d$-tuple
 that is subordinate to $\calv$, then $\calv$ is an $\A$-spectral set for $T$.
 \end{defin}
 
 The von Neumann property is closely related to the extension property.
 The following theorem was proved for the bidisk in \cite{agmcvn}.
 \bt
 Let $\calv$ be a holomorphic subvariety of $\D^d$, and $\A$ a subalgebra of $\hiv$.
 Then $\calv$ has the $\A$ von Neumann property if and only if it has the $\A$ extension property.
 \et
 \bp
 One direction is easy. Suppose $\calv$ has the $\A$ von Neumann property, and 
 $T$ is an  And\^o $d$-tuple
 that is subordinate to $\calv$.
 Let $f \in \A$. By the extension property, there is a function $g \in H^\infty(\D^d)$ that extends $f$ and has the same norm, and $f(T) = g(T)$.
 Since $\sigma(T) \subseteq \D^d$, we can approximate $g$ uniformly on a neighborhood of $\sigma(T)$ by
 polynomials $p_n$ with $\| p_n \|_{H^\infty(\D^d)} \leq \| g \|_{H^\infty(\D^d)}$. Therefore
 \[
 \| f(T) \| = \lim \| p_n (T) \| \leq \| g \|_{H^\infty(\D^d)} = \| f \|_\calv .
 \]
 
 To prove the other direction, 
 let $\Lambda$ be a finite set in $\D^d$, with say $n$ elements $\{ \l_1, \dots, \l_n \}$.
 Let $\KL$ denote the set of $n$-by-$n$ positive definite matrices $K$ that have $1$'s down the diagonal,
 and satisfy
 \[
 [ (1 -  w_i \bar w_j)K_{ij} ] \ \geq \ 0 
 \]
 whenever there is a function $\phi$ in the closed unit ball of $H^\i(\D^d)$ that has $\phi(\lambda_i) = w_i$.
 We shall need the following result, which was originally proved by E. Amar \cite{am77}; see also
 \cite{nak90}, \cite{clw92},  \cite{tw09}, \cite[Thm. 13.36]{ampi}.
 \bt
 \label{thmh2}
Let $\Lambda = \{ \l_1, \dots, \l_n \} \subset \D^d$ and $\{ w_1, \dots , w_n \} \subset \C$.
There exists a function $\phi$ in the closed unit ball of $H^\infty(\D^d)$ that maps each $\l_i$ to the
corresponding $w_i$ if and only if
\[
 [ (1 -  w_i \bar w_j)K_{ij} ] \ \geq \ 0 \quad \forall \ K \in \KL .
 \]
 \et
 
 Suppose $\calv$ has the $\A$ von Neumann property but not the $\A$ extension property.
Then there is some $f \in \A$ with  $\| f \|_\calv = 1$ but no extension of norm $1$ to $\D^d$.
There must be a finite set $\Lambda$ and a number $M > 1$ so that every function $\phi$ in $H^\i(\D^d)$
that agrees with $f$ on $\Lambda$ has $\| \phi \| \geq M$. (Otherwise by normal families one would get an extension
of $f$ of norm one).

Let $\Lambda =  \{ \l_1, \dots, \l_n \}$, and let $w_i := f(\l_i)$ for each $i$. By Theorem~\ref{thmh2}, there
exists some $K \in \KL$ so that
\be
\label{eqh1}
 [ (1 -  w_i \bar w_j)K_{ij} ]  \quad {\rm is\ not\ positive\ semidefinite.}
 \ee
Choose unit vectors $v_i$ in $\C^n$ so that 
\[
\langle v_i , v_j \rangle \ = \ K_{ij} .
\]
(This can be done since $K$ is positive definite).
For each point $\l_i$, let its coordinates be given by $\l_i = (\l_i^1, \dots, \l_i^d)$.
Define $d$ commuting matrices $T$ on $\C^n$ by
\[
T_j v_i \= \lambda_i^j v_j 
\]
Then $T$ is an And\^o $d$-tuple, because if $p$ is a polynomial of norm $1$ on $\D^d$, then 
\[
p(T) : v_i \mapsto p(\lambda) v_i ,
\]
so
\begin{eqnarray*}
\langle (1 - p(T)^* p(T) ) \sum c_i v_i, \sum c_j v_j \rangle
& \ = \ &
\sum c_i \bar c_j (1 - p(\l_i) \overline{ p(\l_j)} \langle v_i , v_j \rangle \\
& = & \sum c_i \bar c_j (1 - p(\l_i) \overline{ p(\l_j)} K_{ij} .
\end{eqnarray*}
This last quantity is positive since $K \in \KL$, so $p(T)$ is a contraction, as claimed.
But since $\calv$ is assumed to have the $\A$ von Neumann property, this means that $f(T)$ is also
a contraction, so $I - f(T)^* f(T) \geq 0$.
But
\[
\langle (1 - f(T)^* f(T) ) \sum c_i v_i, \sum c_j v_j \rangle
 \ = \ 
\sum c_i \bar c_j (1 - w_i \overline{ w_j})  K_{ij} ,
\]
and if this is non-negative for every choice of $c_i$ we contradict \eqref{eqh1}.
 
 \ep

{\bf Acknowledgments.} The first author is very grateful to Maciej Denkowski for fruitful discussions.

\bibliography{references}

\end{document}

%% file: def_latex3.tex
\usepackage{latexsym}
\usepackage{amssymb}
\usepackage{euscript}

\let\cal=\mathcal      

\def\mcc{M\raise.5ex\hbox{c}C}
\def\mccarthy{M\raise.5ex\hbox{c}Carthy}

\def\eg{{\it e.g. }}
\def\ie{{\it i.e. }}

\def\h{{\cal H}}

\def\z{\zeta}

\def\vare{\varepsilon}

\let\i=\infty

\def\={\ = \ }

\def\A{{\cal A}}

\def\C{\mathbb C}

\def\T{\mathbb T}
\def\D{\mathbb D}
\def\B{\mathbb B}

\def\be{\setcounter{equation}{\value{theorem}} \begin{equation}}
\def\ee{\end{equation} \addtocounter{theorem}{1}}
\def\beq{\begin{eqnarray*}}
\def\eeq{\end{eqnarray*}}
\def\se{\setcounter{equation}{\value{theorem}}} 
\def\att{\addtocounter{theorem}{1}}
\def\vs{\vskip 5pt}

\def\bp{{\sc Proof: }}
\def\ep{{}{\hfill $\Box$} \vskip 5pt \par}

\def\bl{\begin{lemma}}
\def\el{\end{lemma}}
\def\bt{\begin{theorem}}
\def\et{\end{theorem}}
\def\bprop{\begin{prop}}
\def\eprop{\end{prop}}
\def\bd{\begin{definition}}
\def\ed{\end{definition}}
\def\br{\begin{remark}}
\def\er{\end{remark}}
\def\bexer{\begin{exercise}}
\def\eexer{\end{exercise}}

\newtheorem{theorem}{Theorem}[section]
\newtheorem{proposition}[theorem]{Proposition}
\newtheorem{lemma}[theorem]{Lemma}
\newtheorem{corollary}[theorem]{Corollary}

\newtheorem{definition}[theorem]{Definition}

\newtheorem{question}[theorem]{Question}